\documentclass[11pt]{article}

\usepackage{hyperref}

\usepackage{amsfonts, amsmath, amssymb, amsthm}
\usepackage{a4}

\usepackage{upgreek}

\newcommand{\VozmozhnyPeremeny}{}
\def\VozmozhnyPeremeny#1{#1}

\newtheorem{theorem}{Theorem}

\theoremstyle{definition}

\newtheorem{dfn}{Definition}
\newtheorem{xmp}{Example}

\theoremstyle{remark}
\newtheorem*{remark}{Remark}
\newtheorem*{ir}{Important remark}

\author{I.G. Korepanov}
\title{An integral bilinear form and related forms on abelian groups as hexagon cocycles}
\date{March 2018}

\begin{document}

\sloppy

\maketitle

\begin{abstract}
Hexagon relations are algebraic realizations of four-dimen\-sional Pachner moves, and there are hexagon relations admitting nontrivial cohomologies and leading thus to piecewise linear (PL) 4-manifold invariants. We show that some---but not all!---of the known nontrivial cohomologies can be obtained from a single integral bilinear form corresponding to a PL 4-manifold by using a Frobenius homomorphism for a half of `color' variables (or different Frobenius homomorphisms for both halves). This form can be regarded as a sophisticated analogue of the manifold's intersection form.
\end{abstract}

\section{Introduction}\label{s:i}

\subsection{Generalities}

A triangulation of a piecewise linear (PL) 4-manifold can be transformed into its any other triangulation by a finite sequence of \emph{Pachner moves}~\cite{Pachner,Lickorish}. Hence, it is natural to expect that a PL 4-manifold \emph{invariant} can be constructed if we have an \emph{algebraic realization} of Pachner moves---informally speaking, such formulas whose structure corresponds to these moves naturally. Such formulas are often called \emph{hexagon relations} (or, more generally, $(n+2)$-gon relations for $n$-manifolds, see, for instance,~\cite{DM}).

A fruitful version of hexagon relation has been proposed in~\cite{bosonic,cubic}. In it, a \emph{two-component} `color'~$(x_t,y_t)$ is ascribed to every tetrahedron~$t$---3\nobreakdash-face of a triangulation. Thus, there appear ten variables on the 3-faces of each pentachoron (4\nobreakdash-simplex)~$u$ and it is required that these ten variables obey \emph{five} relations. Moreover, the variables $x_t,y_t$ in~\cite{bosonic,cubic} belonged to a field, and the relations were linear. As we will see below in this paper, this construction can work productively also, at least, for modules over ring~$\mathbb Z$, that is, \emph{abelian groups}.

The most interesting manifold invariants appear if we use \emph{cohomology} of our hexagon relations (in perfect analogy with the fact that powerful invariants of \emph{knots} and their higher analogues are obtained from \emph{quandle} cohomology~\cite{CKS}). In~\cite{cubic}, we have calculated some `polynomial' cohomologies; one can see that they are very nontrivial, and give nontrivial manifold invariants.

The nature of these polynomial cohomologies looks, at the moment, quite mysterious. At the same time, there is a hope that it can be clarified; compare, for instance, the study of polynomial \emph{quandle} cocycles in~\cite{CEF} and references therein. A natural desire is also to relate our cohomologies to known algebraic structures appearing in the study of 4-manifolds, or at least find parallels between them.

\subsection{What we do in the present paper}

We bring to light some parallels between one cocycle appearing in our work and the \emph{intersection form} of 4-manifolds (see textbooks~\cite{GS,Scorpan}). We construct a similar $\mathbb Z$-bilinear form in the framework of our theory. Further, we show that \emph{some} of nontrivial polynomial hexagon cocycles in finite characteristics found in~\cite{cubic} can be derived from this single $\mathbb Z$-bilinear form---to be exact, its modifications involving finite fields---using manipulations with Frobenius homomorphisms.

\begin{ir}
Not all polynomial cocycles could be obtained that way, at least by now!
\end{ir}

\subsection{Notational conventions}\label{ss:cnvs}

Typically, we denote 2\nobreakdash-simplices---triangles---by the letter~$s$, 3\nobreakdash-simplices---tetrahedra---by the letter~$t$, and 4\nobreakdash-simplices---pentachora---by the letter~$u$.

All vertices of any triangulated object are assumed to be \emph{numbered}. A triangle $s=ijk$ has vertices whose numbers are $i$, $j$ and~$k$.

Moreover, when we denote a simplex by its vertices, these go, by default, in the \emph{increasing} order. For the above triangle~$s$, this means that $i<j<k$. 

\subsection{The content of the paper by sections}

Below,
\begin{itemize}\itemsep 0pt
 \item in Section~\ref{s:h}, we recall, very briefly, the notion of \emph{permitted colorings} of 3-faces satisfying the \emph{full hexagon}. One small new moment is that we formulate all this in terms of general abelian groups,
 \item in Section~\ref{s:c}, we introduce our $\mathbb Z$-bilinear hexagon cocycle,
 \item in Section~\ref{s:inv}, we describe two kinds of manifold invariants: the `probabilities' of the values that our bilinear `action' can take on a manifold, and just the action itself understood as an \emph{integral} bilinear form taken to within invertible $\mathbb Z$-linear transformations and adding zero direct summands,
 \item in Section~\ref{s:ff}, we recall polynomial cocycles over finite fields from~\cite{cubic} and show how \emph{some} of these can be reproduced using our bilinear form and `Frobenius tricks',
 \item in Section~\ref{s:if}, we explain why the usual intersection form of a 4-manifold can be seen as an analogue of our $\mathbb Z$-bilinear `action',
 \item and finally, in Section~\ref{s:d}, we give a very brief discussion of our results and some related intriguing questions.
\end{itemize}

\section{Hexagon relation}\label{s:h}

\subsection{Colorings by an abelian group}\label{ss:col}

Let $G$ be an abelian group. We will color tetrahedra~$t$---that is, triangulation 3-faces---by \emph{pairs} $(x_t,y_t)$ of elements $x_t,y_t\in G$, and call this simply \emph{$G$-coloring}. For a pentachoron~$u=ijklm$, we introduce two vector columns of height~$5$, corresponding to its 3-faces going in the inverse lexicographic order:
\begin{equation}\label{xy}
\mathsf x_u= \begin{pmatrix} x_{jklm}\\ x_{iklm}\\ x_{ijlm}\\ x_{ijkm}\\ x_{ijkl} \end{pmatrix},\quad
\mathsf y_u= \begin{pmatrix} y_{jklm}\\ y_{iklm}\\ y_{ijlm}\\ y_{ijkm}\\ y_{ijkl} \end{pmatrix}.
\end{equation}
We also take the following matrix with integer entries from~\cite[Eqs.~\VozmozhnyPeremeny{(5)~and~(6)}]{cubic}:
\[
\mathcal R = \begin{pmatrix}0 & -2 & 1 & 1 & -2\\
0 & -1 & 0 & 1 & -1\\
-1 & 2 & -2 & 0 & 1\\
-1 & 3 & -2 & -1 & 2\\
0 & 1 & -1 & 0 & 0\end{pmatrix}.
\]

\begin{dfn}
A coloring of (the 3-faces of) a pentachoron~$u$ is called \emph{permitted} if the following relation holds:
\[
\mathsf y_u=\mathcal R\mathsf x_u. 
\]
\end{dfn}

\begin{dfn}
A coloring of a triangulated piecewise linear 4-manifold~$M$ is called permitted if it induces permitted colorings on all its pentachora.
\end{dfn}

Permitted colorings of the initial and final configurations of any Pachner move are in good correspondence with each other. We say that they satisfy \emph{full set theoretic hexagon}. For exact formulations, the reader is referred to~\cite[\VozmozhnyPeremeny{Sections 2 and~3}]{cubic}.

\begin{remark}
The fact that we are using here an arbitrary abelian group instead of a field in~\cite{cubic} brings nothing significantly new into our definitions and reasonings.
\end{remark}

\subsection{Double colorings}\label{ss:dc}

One case of special importance is \emph{double coloring}. This is, by definition, a coloring in the sense of Subsection~\ref{ss:col} with $G$ being a direct sum of two abelian groups:
\[
G = A \oplus B.
\]
In this case, we will use the following notations for the pairs of elements of each group:
\[
(x_t,y_t)\in A^{\oplus 2},\quad (\xi_t,\eta_t)\in B^{\oplus 2}.
\]
Similarly to~\eqref{xy}, we introduce four columns:
\[
\mathsf x_u= \begin{pmatrix} x_{jklm}\\ x_{iklm}\\ x_{ijlm}\\ x_{ijkm}\\ x_{ijkl} \end{pmatrix},\quad
\mathsf y_u= \begin{pmatrix} y_{jklm}\\ y_{iklm}\\ y_{ijlm}\\ y_{ijkm}\\ y_{ijkl} \end{pmatrix},\quad
\upxi_u= \begin{pmatrix} \xi_{jklm}\\ \xi_{iklm}\\ \xi_{ijlm}\\ \xi_{ijkm}\\ \xi_{ijkl} \end{pmatrix}\text{ \ \ and \ \ }
\upeta_u= \begin{pmatrix} \eta_{jklm}\\ \eta_{iklm}\\ \eta_{ijlm}\\ \eta_{ijkm}\\ \eta_{ijkl} \end{pmatrix},
\]
and a permitted double coloring of a pentachoron is of course such that
\begin{equation}\label{ye}
\mathsf y_u=\mathcal R\mathsf x_u, \qquad \upeta_u=\mathcal R\upxi_u.
\end{equation}

\section[A hexagon cocycle in the form of a $\mathbb Z$-bilinear form]{A hexagon cocycle in the form of a $\boldsymbol{\mathbb Z}$-bilinear form}\label{s:c}

\begin{theorem}
The following bilinear form:
\begin{multline}\label{Phi}
\Phi_u(\mathsf x_u, \upxi_u) = (x_{jklm}+y_{jklm}) \otimes (\xi_{ijkl}+\eta_{ijkl}) \\
 = (x_{jklm}-2x_{iklm}+x_{ijlm}+x_{ijkm}-2x_{ijkl}) \otimes
(\xi_{iklm}-\xi_{ijlm}+\xi_{ijkl})
\end{multline}
is a nontrivial hexagon 4-cocycle. Here $u=ijklm$, and the form~$\Phi$ depends thus on a pair ``permitted $A$-coloring, permitted $B$-coloring'' (relations~\eqref{ye} are of course implied) and takes values in the abelian group $A\otimes B$.
\end{theorem}

\begin{proof}
Direct calculation. 
\end{proof}

\begin{ir}
Recall (Subsection~\ref{ss:cnvs}) that our notational conventions imply that $i<j<k<l<m$, so $ijkl$ may be called the \emph{front} 3-face of pentachoron $u=ijklm$, while $jklm$---its \emph{rear} 3-face. The first line of~\eqref{Phi} shows that $\Phi_u$ is the product of two quantities belonging to these two 3-faces. This is especially interesting when compared with the manifold's \emph{intersection form}, see Section~\ref{s:if} below.
\end{ir}

\section{Invariants}\label{s:inv}

\subsection{The action}

For a triangulated oriented 4-manifold~$M$, maybe with boundary, we introduce the following \emph{action}, depending on a permitted coloring of~$M$:
\begin{equation}\label{S}
S = \sum_u \epsilon_u \Phi_u(\mathsf x_u, \upxi_u),
\end{equation}
where the sum goes over all triangulation pentachora~$u$, and $\epsilon_u=1$ if the orientation of~$u$ determined by the increasing order of its vertices coincides with its orientation induced from~$M$, and $\epsilon_u=-1$ otherwise.

Below, in Subsections \ref{ss:subgr} and~\ref{ss:ibf}, we specialize this action for two interesting cases.

\subsection{Subgroups of permitted colorings and probabilities of action values}\label{ss:subgr}

\begin{theorem}\label{th:P}
Let $A$ and~$B$ be as in Subsection~\ref{ss:dc} and, moreover, let them be \emph{finite}, and let a subgroup $F\subset A\oplus B$ be given. For each permitted $F$-coloring, action~$S$ takes some value $v\in A\otimes B$. Define the \emph{probability} of value~$v$ as
\begin{equation}\label{P}
\mathrm P(v) = \frac{\# (S=v)}{\# (\text{all permitted colorings})},
\end{equation}
where $\#$ means the cardinality of a set; so, $\# (S=v)$ is the number of those permitted colorings where $S=v$, compare~\cite[\VozmozhnyPeremeny{(19)}]{cubic}.

Then, the probabilities\/ $\mathrm P(v)$ for all\/ $v\in A\otimes B$ are invariants of the piecewise linear manifold~$M$.
\end{theorem}

\begin{proof}
It repeats the proof of \cite[Theorem~\VozmozhnyPeremeny{4(i)}]{cubic}.
\end{proof}

We will see in Section~\ref{s:ff} how to choose subgroup~$F$ in order to obtain \emph{some} of polynomial actions in~\cite{cubic}.

\subsection{Integral bilinear form}\label{ss:ibf}

Set now simply $A=B=\mathbb Z$, and the subgroup $F=G=A\oplus B$. Then action~$S$ given by~\eqref{S} becomes an integral bilinear form.

\begin{theorem}\label{th:Z}
Let bilinear form~$S$ (depending on a pair of permitted integral colorings) be defined according to~\eqref{S},~\eqref{Phi} with integer variables $x_t$ and~$\xi_t$. Then $S$, taken to within a zero direct summand and an invertible $\mathbb Z$-linear transformation, is a piecewise linear 4-manifold invariant.
\end{theorem}

\begin{proof}
First, we analyze what happens under any Pachner move. Everything goes in, essentially, the same way as in~\cite[Section~\VozmozhnyPeremeny{3}]{cubic} (although, there was a \emph{field} in~\cite{cubic} instead of our current~$\mathbb Z$):
\begin{itemize}\itemsep 0pt
\item for a Pachner move 3--3, there is a bijective---and certainly $\mathbb Z$-linear---correspondence between the permitted $\mathbb Z$-colorings before and after this move,
\item for a move 2--4, there appears one additional `$\mathbb Z$-degree of freedom': to each permitted coloring before the move there correspond colorings after the move, parameterized by~$\mathbb Z$,
\item for a move 1--5, there appear, in a similar way, four additional `$\mathbb Z$-degrees of freedom',
\item and for all Pachner moves, the value of~$S$ remains the same, after applying this correspondence, due to the cocycle property.
\end{itemize}

Finally, if we choose a different basis in the space (to be exact, free $\mathbb Z$-module) of permitted colorings, this will correspond, of course, to an invertible $\mathbb Z$-linear transformation applied to the form~$S$.
\end{proof}

\begin{theorem}\label{th:sym}
For a closed oriented triangulated 4-manifold~$M$, the form~$S$ of Theorem~\ref{th:Z} is symmetric.
\end{theorem}

\begin{proof}
To see this, we subtract from~\eqref{Phi} the same expression, but with interchanges $x_t\leftrightarrow \xi_t$, \ $y_t\leftrightarrow \eta_t$. It turns out that the result is a coboundary, in the sense of hexagon cohomology. Below we write this coboundary even in two different ways:
\begin{align*}
(x_{jklm}+y_{jklm}) & (\xi_{ijkl}+\eta_{ijkl}) - (\xi_{jklm}+\eta_{jklm})(x_{ijkl}+y_{ijkl}) \\
 & = (x_{jklm}+y_{jklm})\eta_{jklm}-(x_{iklm}+y_{iklm})\eta_{iklm}+\dots\\
 & = -y_{jklm}(\xi_{jklm}+\eta_{jklm})+y_{iklm}(\xi_{iklm}+\eta_{iklm})-\dots\,.
\end{align*}
Here the omission points mean, in both cases, three obvious similar summands, corresponding to the remaining 3-faces $ijlm$, $ijkm$ and~$ijkl$, and coming with alternating signs.

Thus, the difference between $S$ and the same~$S$, but with the mentioned `Latin--Greek' interchanges, consists of summands belonging to triangulation tetrahedra, and, moreover, for a closed~$M$ they all cancel away.
\end{proof}

\begin{remark}
Of course, if we know the integral bilinear form~$S$ of this Subsection, we can get also all the forms mentioned above and involving abelian groups, by using the tensor product operation. The reader will hopefully have no difficulty in writing the exact formulas.
\end{remark}

\section{Particular case: polynomial cocycles over finite fields}\label{s:ff}

\subsection[Introducing dependence of~$\xi_t$ on~$x_t$ using Frobenius homomorphism]{Introducing dependence of~$\boldsymbol{\xi_t}$ on~$\boldsymbol{x_t}$ using Frobenius homomorphism}\label{ss:F}

We continue to use the notations of Subsection~\ref{ss:dc}. Let $A=B=\mathbb F_{p^n}$, that is, the finite field of $p^n$ elements, and let $F\subset A\oplus B$ be generated by such pairs $(a,b)=(x_t,\xi_t)$ that $b=a^{p^m}$, where $m\in \mathbb N_0=\{0,1,2,\dots\}$. Also, we replace the tensor product in the definition~\eqref{Phi} of~$\Phi_u$ by the usual multiplication in~$\mathbb F_{p^n}$. Then, it is easy to see that the so defined~$\Phi_u$ yields in the standard way the action $S=\sum_u \epsilon_u \Phi_u$ whose probabilities~\eqref{P} are manifold invariants, and this~$\Phi_u$ is a polynomial of degree~$p^m+1$ in the five variables~$x_t$, where $t\subset u$.

\begin{ir}
Probabilities $\mathrm P(v)$ for $v\in A\otimes B$ may be more informative than those defined according to the previous paragraph (and also used in~\cite{cubic}), because $A\otimes B=A\otimes_{\mathbb Z} B$ is typically greater that just the field $\mathbb F_{p^n} = A\otimes_{\mathbb F_{p^n}} B$ where the action defined in the previous paragraph takes values.
\end{ir}

\begin{xmp}
Let $p=2$, \ $n$~any natural number, and $m=0$. Then \eqref{Phi} becomes a homogeneous quadratic polynomial in the variables~$x_t$ in characteristic~2, namely
\[
(x_{jklm}+x_{ijlm}+x_{ijkm})(x_{iklm}+x_{ijlm}+x_{ijkl}).
\]
This is the same polynomial as appears in~\cite[Eq.~\VozmozhnyPeremeny{(22)}]{cubic}.
\end{xmp}

\begin{xmp}
Let again $p=2$, \ $n$~any natural number, but now $m=1$. Then \eqref{Phi} becomes a homogeneous \emph{cubic} polynomial, namely
\[
(x_{jklm}+x_{ijlm}+x_{ijkm})(x_{iklm}^2+x_{ijlm}^2+x_{ijkl}^2).
\]
This is the same polynomial as appears in~\cite[Eq.~\VozmozhnyPeremeny{(24)}]{cubic}.
\end{xmp}

\subsection{Example of a cocycle that could not be obtained by the Frobenius trick}

The following cubic cocycle in characteristic~2 could not (as yet) be obtained by the author using the above `Frobenius' or any similar trick: 
\begin{align}
x_{iklm}x_{ijkm}x_{ijkl} + x_{iklm} & x_{ijlm}x_{ijkl} + x_{jklm}x_{ijlm}x_{ijkl} \nonumber \\
 &  + x_{jklm}x_{ijlm}x_{ijkm} + x_{jklm}x_{iklm}x_{ijkm}. \label{cubic-first}
\end{align}
This cocycle appears, and in more elegant notations, in~\cite[Eq.~\VozmozhnyPeremeny{(23)}]{cubic}.

\subsection{Double Frobenius trick}\label{ss:FF}

One obvious generalization of the `Frobenius trick' of Subsection~\ref{ss:F} is to do it on both `Latin' and `Greek' variables. Namely, let now introduce variables~$\mathfrak x_t$ over the field $A=B=\mathbb F_{p^n}$, and let $F\subset A\oplus B$ consist of pairs $(a,b)=(x_t,\xi_t)$ parameterized by variables $c=\mathfrak x_t$ in such way that $a=c^{p^{m_1}}$ and $b=c^{p^{m_2}}$, where $m_1,m_2\in \mathbb N_0$.

For a single pentahoron~$u$, we consider the five $\mathfrak x_t$, \ $t\subset u$, as independent variables, and for the whole manifold~$M$, variables~$\mathfrak x_t$ are assumed to have the same linear dependencies as either~$x_t$ or~$\xi_t$. Recall that these dependencies arise due to~\eqref{ye} and the fact the pair $(x_t,y_t)$ or $(\xi_t,\eta_t)$ for any tetrahedron~$t$ must be the same regardless of the pentachoron~$u\supset t$ where we used formula~\eqref{ye} (and there are of course two such~$u$ for a non-boundary~$t$).

\begin{xmp}
Let again $p=2$, \ $n$~any natural number, $m_1=1$ and $m_2=2$. Then \eqref{Phi} becomes a homogeneous polynomial of the \emph{sixth} degree, namely
\[
(\mathfrak x_{jklm}^2+\mathfrak x_{ijlm}^2+\mathfrak x_{ijkm}^2)(\mathfrak x_{iklm}^4+\mathfrak x_{ijlm}^4+\mathfrak x_{ijkl}^4).
\]
The reader will hopefully have no problem in finding this polynomial also in an unnumbered formula in~\cite[\VozmozhnyPeremeny{Subsection~5.1}]{cubic}.
\end{xmp}

Still, as of now, neither cocycle~\eqref{cubic-first} nor many other cocycles from~\cite{cubic} could be obtained using anything like these tricks.

\section{Intersection form as a simple analogue of our construction}\label{s:if}

This time, we color \emph{2-faces} $s=ijk$, \ $i<j<k$ (remember Subsection~\ref{ss:cnvs}), by elements~$x_{ijk}\in \mathbb Z$.

\emph{Permitted colorings} are such where values $x_{ijk}$ make a \emph{cocycle}, that is,
\[
x_{ijk} - x_{ijl} + x_{ikl} - x_{jkl} = 0
\]
for any triangulation tetrahedron~$ijkl$.

In analogy with Subsection~\ref{ss:dc}, we consider now \emph{double} colorings $(x_s,\xi_s)$, where $\xi_{ijk}\in \mathbb Z$ also make a cocycle. A simple calculation shows that the bilinear form
\[
\varphi_{ijklm} = x_{ijk} \xi_{klm}
\]
well-known from the theory of the cup product, is (of course!) a hexagon cocycle. We introduce then the following simple `action' (with the same~$\epsilon_u$ as in~\eqref{S}):
\[
S=\sum_u \epsilon_u \varphi_u,
\]
which is nothing but the well-known \emph{intersection form}~\cite{GS,Scorpan} of a 4-manifold.

\section{Discussion of results}\label{s:d}

Our integral bilinear form in Subsection~\ref{ss:ibf} may actually simply \emph{coincide} with the usual intersection form. At the moment, there are neither calculations that could disprove this conjecture, nor a theory that would prove it.

Anyhow, what looks most interesting is that out form, together with `Frobenius tricks' from Section~\ref{s:ff}, yields \emph{some} but not all polynomial cocycles found in~\cite{cubic}. So, the nature of the remaining cocycles looks very intriguing.

\end{document}